\documentclass{amsart}

\usepackage{amscd}
\usepackage{amssymb}

\usepackage{color}   
\usepackage{hyperref}
\hypersetup{
  colorlinks=true, 
   linkcolor=red,  
}

\newtheorem{theorem}{Theorem}

\newtheorem{lem}[theorem]{Lemma}
\newtheorem{cor}[theorem]{Corollary}
\newtheorem{prop}[theorem]{Proposition}

\theoremstyle{definition}

\newtheorem{defn}[theorem]{Definition}

\theoremstyle{remark}

\newtheorem{rmk}[theorem]{Remark}
\newtheorem*{rmk*}{Remark}

\DeclareMathOperator{\e}{\mathbb{E}}
\def\p{\mathbb{P}}
\def\P{\mathbb{P}}
\def\BR{{\mathbb{R}}}
\def\BZ{{\mathbb{Z}}}
\def\t{(t_1,t_2,\ldots)}
\def\range{{\mathcal{B}}}
\def\b{(b_0,b_1,b_2,\ldots)}
\def\colln{{\mathcal{C}_n}}
\DeclareMathOperator{\var}{Var}
\def\cF{\mathcal{F}}

\newcommand{\la}{\lambda}
\newcommand{\eand}{\quad\text{and}\quad}

\DeclareMathOperator{\Var}{Var}

\newcommand{\BX}{\mathbf{X}}

\newcommand{\s}{\sigma}

\newcommand{\E}{\mathbb{E}}

\begin{document}

\title[Asymptotic distribution for the birthday problem]{Asymptotic distribution for the birthday problem with multiple coincidences, via an embedding of the collision process}
\author{R. Arratia \and S. Garibaldi \and J. Kilian}
\address{Arratia: Department of Mathematics, University of Southern California, Denney Research Building DRB 155, Los Angeles, CA 90089-1113}
\email{rarratia \text{at} usc.edu}

\address{Garibaldi: Institute for Pure and Applied Mathematics, UCLA, 460 Portola Plaza, Box 957121, Los Angeles, California 90095-7121, USA}
\email{skip \text{at} member.ams.org}

\address{Kilian: IDA Center for Communications Research, 805 Bunn Dr., Princeton, NJ 08540}



\begin{abstract}
We study the random variable $B(c,n)$, which counts the number of balls that must be thrown into $n$ equally-sized bins in order to obtain $c$ collisions.  The asymptotic expected value of $B(1,n)$ is the well-known $\sqrt{n\pi/2}$ appearing in the solution to the birthday problem; the limit distribution and asymptotic moments of $B(1,n)$ are also well-known.  We calculate the distribution and moments of $B(c,n)$ asymptotically as $n$ goes to $\infty$ and $c = O(n)$.

Our main tools are an embedding of the collision process, realizing the process as a deterministic function of the standard Poisson process, and a central limit result by R\'enyi.
\end{abstract}

\maketitle

\tableofcontents
 
\section{Introduction}

\subsection{Scientific motivation}

Imagine throwing balls into $n$ equally-sized bins.  One  \emph{collision} occurs whenever a ball lands in a bin that already contains a ball, so that a bin containing $k$ balls contributes $k-1$ to the total number of collisions.  This notion of collision is relevant for hash tables in computer science as in \cite[\S6.4]{Knuth3} and for cryptology as in \cite{KuhnStruik}; it is the definition of collision from \cite[\S3.3.2I, p.~69]{Knuth2}.  

We study the random variable $B(c,n)$, which counts the number of balls you throw into $n$ bins to produce $c$ collisions. 
The classic birthday problem as described in \cite[p.~33]{Feller}, \cite[Problem 31]{Mosteller}, and \cite{Strogatz} asks for the median of $B(1, 365)$.
We define $B(0,n)=0$.  We have: 
\begin{equation} \label{big.c}
c+1 \le B(c,n) \le c+n \quad \text{for $c = 1, 2, \ldots$.}
\end{equation}
For $n = 1$, this forces  $B(c,1) = c+1$, which makes sense, since with a single bin, the first ball does not make a collision but all subsequent balls do. 


The variable $B(1,n)$ appears in the standard birthday problem for a year with $n$ days, so it has already been well studied.  Indeed, 
\begin{equation} \label{birthday}
B(1,n)/\sqrt{n} \Rightarrow L, \quad \text{and} \quad  \e B(1,n) \sim \sqrt{n} \  \e  L  = 
\sqrt{\frac{\pi}2 \, n \, }  ,
\end{equation}
where the limit distribution of $L$ is attributed to $L$ord Rayleigh, with $\P(L > t) = \exp(-t^2/2)$ for $t > 0$, and $\e L = \sqrt{\pi/2}$, see \cite[11.3]{RS}, \cite[(3.7), (3.10)]{Harris}, or \cite[Example 3.2.5]{Durrett}.  
On the other hand,  for $c_n \to \infty$ with $c_n=o(n^{1/4})$ Kuhn and Struick \cite[p.~221]{KuhnStruik} show that
$$  \e B(c_n,n)  \sim \sqrt{2 \, c_n \, n \,},
$$
which matches \eqref{birthday} apart from the coefficient of $c_n \,n$ inside the square root changing \emph{from} $\pi/2$ \emph{to} $2$.
Indeed, the impetus for this paper was the desire to explain how $\pi/2$ changes to $2$, and the title of our initial writeup was ``From $\pi/2$  to $2$:  $\pi/2, 9\pi/16, 75 \pi/128, 1225\pi/2048,\ldots,2$''.  See subsection \ref{sect gamma} for more details.

\subsection{Quick survey of the contents}
We consider $B(c,n)$, the number of balls that must be thrown into $n$ bins, in order to get a specified number $c$ of collisions.
To investigate this, we consider in Section \ref{sect embed} an embedding of the collision process into a standard Poisson process;  the embedding may be of interest in its own right, and we give a variety of almost sure uniform error bounds, culminating in Theorem \ref{thm uniform}.  Even the simplest process convergence, Theorem \ref{thm rateless}, which holds for all outcomes $\omega \in \Omega$, implies a process distributional limit, Corollary \ref{process converge distribution},  which in turn gives the limit one-dimensional distributional limit:  Corollary \ref{marginal converge distribution}, stating that for fixed $c$,  $B(c,n)/\sqrt{n} \Rightarrow \sqrt{2T_c}$, the Chi distribution with $2c$ degrees of freedom.   This in turn, combined with a uniform integrability estimate in Section \ref{sect ui},  gives the asymptotic 
mean and variance of $B(c,n)$ for fixed $c$, with details given in Section \ref{sect gamma}, in particular \eqref{first asymptotic moment}, \eqref{gamma formula}, and \eqref{gamma formula var}.

We are mainly interested in the case where $c_n = o(n)$, because that is the case relevant for applications as in \cite{KuhnStruik}.  However, in analyzing the variance of $B(c,n)$,  for  $c \approx n^a$  with $1/2 \le a <1$,  our embedding is
not an appropriate tool, and we were forced to work with duality and R\'enyi's central limit theorem for the number of empty boxes.  This duality also \emph{easily} handles the ``central region'',
corresponding to   $c_n/n \to \alpha_0 \in (0,\infty)$,  hence we include such results in sections \ref{sect dual}--\ref{moments.duality}, such as Theorem \ref{MT} and Corollary \ref{cor var}.  Note that the results in the last three sections concern a centered distribution $B(c,n) - \beta(c,n)$.  Our penultimate result, Corollary \ref{cor var}, determines the moments of the centered distribution and the variance of $B(c,n)$ over a large range of choices for $c$.  This, combined with the results of Section \ref{sect gamma} extend the result from \cite{KuhnStruik} to a much larger regime.

Our main results are new, despite the substantial existing literature on other occupancy problems, such as \cite{urns}, \cite{KSC}, \cite{Holst:survey}, \cite{Holst}, \cite{GHP}, etc., and other work on $B$ such as \cite{CamarriPitman}.  (Although Theorem \ref{thm uniform} could be recovered over a smaller regime by using Poisson approximation as in \cite{AGG} or \cite{BHJ}, cf.~Remark \ref{AGG} below.)

\begin{table}[hbt]
\begin{tabular}{||cccc||} \hline
regime&convergence&uniform integrability&moments \\ \hline
fixed $c$&Theorem \ref{marginal converge distribution}&Lemma \ref{UI}&Corollaries \ref{cor fixed c}, \ref{var} \\
$c = O(n^\alpha)$, $\alpha < 1$&Theorem \ref{thm uniform}&Lemma \ref{UI}&Corollary \ref{cor c grows} \\
$c \to \infty$, $c/n \to \alpha_0 \in [0, \infty)$&Theorem \ref{MT}&Lemma \ref{ui2}&Corollaries \ref{cor c grows}, \ref{cor var}, \ref{on} \\ \hline
\end{tabular}
\caption{Summary of results concerning $B(c, n)$ as $n \to \infty$.  The first two lines deal with the uncentered $B(c, n)$ and the last line with a centered version, $B(c,n) - \beta(c,n)$.}
\end{table}

\section{The classical occupancy process} \label{occ.def}

The classical occupancy \emph{problem} is specified in terms of a fixed number of balls, and a fixed number of equally likely bins.   We choose the notation $b$ balls and $n$ bins, although the notation $n$ balls and $N$ bins, used for example by R\'enyi, is tempting, as it corresponds to the tradition, in statistics, of a sample of size $n$ taken from a population of size $N$.  The classical
occupancy problem starts with independent and identically distributed $X_1,X_2,\ldots$, with $\p(X_t=i)=1/n$ for $i=1$ to $n$ and $t \ge 1$, so that all $n^b$ possible values for $(X_1,\ldots,X_b)$ are equally likely,  and  considers the distributions of $N_0=N_0(b,n)$, the number of empty bins; $I=I(b,n)$, the number of occupied bins; and more generally, for each $k=0,1,2,\ldots$,
the distribution of $N_k=N_k(b,n)$, the number of bins with exactly $k$ balls.     Even at the level of describing the distribution of an individual $N_k(b,n)$, there is much to be said, see for example  \cite{KSC,Renyi:orig,Weiss,Englund,Mikhailov,BG,BGI}.  

As a summary of the notation:
\begin{equation}\label{notations 1}
     N_k(b,n) = \sum_{i=1}^n 1\left(k=\sum_{t=1}^b 1(X_t=i)\right)
\end{equation}
is the number of bins containing exactly $k$ balls, when $b$ \emph{b}alls have been tossed into
$n$ bi\emph{n}s.   As a check:
$$  
  \sum_{k\ge 0} N_k(b,n) = n \eand \sum_{k\ge 0} k \, N_k(b,n)  = b.
$$
The number of \emph{occupied} bins,  when $b$ balls have been tossed into $n$ bins, is
\begin{equation}\label{notation I}
   I(b,n) := n - N_0(b,n) = \sum_{k \ge 1} N_k(b,n),
\end{equation}
and the number of collisions obtained is
\begin{equation}\label{notation C}
   C(b,n) := b - I(b,n)  = \sum_{k \ge 1} (k-1) N_k(b,n).
\end{equation}

The classic occupancy \emph{process} goes a little further:  the number $n$ of bins is fixed, and
balls are tossed in succession, so that the  count of occupied bins,
$I(b,n)$, 
is determined by the locations $X_1,X_2,\ldots,X_b$ of the first $b$ balls, and the entire
process $(I(0,n),I(1,n),I(2,n),\ldots,I(b,n),\ldots)$  is determined by the locations $X_1,X_2,\ldots$
of the balls in $\{1,2,\ldots,n\}$.  Thanks to equally likely bins, 
the process $(I(0,n),I(1,n),I(2,n),\ldots,I(b,n),\ldots)$  also has the structure of a birth process, with
\begin{eqnarray*}
  \p(I(t+1,n)=i \mid I(t,n)=i) &=&  \frac{i}{n}, \\ \p(I(t+1,n)=i+1 \mid I(t,n)=i) &=&  \frac{n-i}{n},
\end{eqnarray*}
and this idea, exploited by R\'enyi \cite{Renyi:orig}, may be considered as the foundation of our embedding, given in Section \ref{sect embed}. 

The collisions process can also be viewed as a birth process, with 
\begin{eqnarray*}
\p(C(b+1,n)=c+1 \mid C(b,n)=c) &=&  \frac{b-c}{n}, \\
  \p(C(b+1,n)=c \mid C(b,n)=c) &=&  \frac{n-(b-c)}{n} .
\end{eqnarray*}

Finally, given the collision counting process, $(C(0,n),C(1,n),C(2,n),\ldots)$, the number of balls needed to get $c$ collisions is defined by \emph{duality}:  for $c=0,1,2,\ldots$,
\begin{equation}\label{define B}
   B(c,n) := \inf \{ b:  C(b,n) \ge c \},
\end{equation}
where, of course, the infimum of the empty set is taken to be $\infty$.

 
\section{The embedding}\label{sect embed}

\subsection{Motivation and informal description}\label{sect informal}

Let $Y$ denote the standard, rate 1 Poisson process.  It has the property that $Y(t)$ is a Poisson random variable with expectation $t$.  Define $T_c$ to be the time of the $c$-th arrival in $Y$.

Let  $f(p)$ be the amount of time one must run the process $Y$ so that, with probability $p$, there is at least one arrival; by standard Poisson process calculations, for $0 \le p <1$, $1 -  p = e^{-f(p)}$.   We extend this to $f\! : [0,1] \to [0,\infty]$ given by
\[ 
f(p) : = {-\log(1-p)} = p + p^2/2 + p^3/3 + \cdots \quad \text{for $0 \le p < 1$}
\] 
and $f(1) := \infty$.
Clearly $f$ is strictly increasing, and maps its domain onto its range.

For fixed positive integer $n$, we now define a coupling of the random variables $B(c,n)$ and $T_c$ for nonnegative integer $c$.  To sample $B(1,n), B(2,n), \ldots,$ $B(c_{\rm max},n)$, we define state variables $i$, $t$ and $c$, all of which are initially set to $0$.  For intuition, $i$ denotes the number of occupied bins, $t$ denotes the amount of time $Y$ has run for, and $c$ denotes the number of collisions.  Our sampling algorithm repeats the following sequence of steps until $c=c_{\rm max}$.
\begin{enumerate}
\renewcommand{\theenumi}{\alph{enumi}}
\item \label{alg1} Run $Y$ for \emph{up to} $f(i/n)$ units of time, stopping immediately if there is an arrival; let $a$ be the first arrival time if there is one.
\item If there is no arrival, add $f(i/n)$ to $t$ and increment $i$ by 1; we call this a ``miss'' step; it corresponds to a ball being thrown without causing a collision.  If there is an arrival, increment $c$, add $a$ to $t$ and put $(B(c,n), T_c) = (c + i, t)$; we call this a ``hit'' step.
\item Return to step \eqref{alg1}.  
\end{enumerate}
Note that conditioned on
the arrivals $(T_1, T_2,...)$, the sampling process described above is deterministic.

Also, the sequence $(B(1, n), B(2,n), \ldots)$ obtained from the sampling algorithm described above has the same distribution as the sequence of the same name defined in terms of throwing balls into bins.  Indeed, if you throw a ball into 1 of $n$ bins, $i$ of which are occupied, the probability of that throw causing a collision is $i/n$, the same as the probability of an arrival in step \eqref{alg1}.

\subsection{Formal description}
To minimize notation, we will take the sample space to be the set of strictly increasing sequences of
strictly positive real numbers, since such a sequence corresponds to the sequence of arrival times in the standard Poisson process:
\begin{equation}\label{def Omega}
\Omega = \{ \omega = \t \in \BR^\infty:  0 < t_1 < t_2 < \cdots \}.
\end{equation}
Of course, in this setup, the random variable $T_c$ is just the $c$-th coordinate, so for
$\omega = \t$, $T_c(\omega)=t_c$.  

\begin{defn}[Formal specification of the embedding] \label{formal}
For any $n=1,2,\ldots$ and $\omega \in \Omega$, we \emph{define} $B(c,n)(\omega)$ and $J(c,n)(\omega)$ 
for $c=0,1,2,\ldots$ recursively, via
$$ 
    B(0,n):=0, \quad J(0,n):=0,
$$
and for $c \ge 1$,   
\begin{multline}
   B(c,n)(\omega) := B(c-1,n)(\omega)  \label{recurse} \\ 
   +  \inf \left\{ j: \left( \sum_{J(c-1,n)(\omega) \le i <  J(c-1,n)(\omega) +j} f(i/n) \right)  \ge \left( T_c(\omega)-T_{c-1}(\omega) \right)\right\}   
\end{multline}
and
\[
J(c,n)(\omega):= B(c,n)(\omega)-c.
\]
\end{defn}

(We think of $J(c,n) := B(c,n)-c$ as the number of
\emph{occupied} bins, when the number of balls tossed is just enough to have formed $c$ collisions.)

We want to separate which properties of our coupling are \emph{deterministic} from  which are
\emph{distributional}.  Hence to serve as the \emph{range} for the coupling, with the notation $\BZ_+ = \{0,1,2,\ldots \}$ for the nonnegative integers, we define
\begin{equation}\label{def range}
  \range = \left\{ \b \in \BZ_+^{\BZ_+}  \! : 0 = b_0 < b_1 < b_2 < \cdots \right\}.
\end{equation}

\begin{theorem}\label{thm embed}
With $\Omega$ given by \eqref{def Omega} and $\range$ given by \eqref{def range}, for every value $n=1,2,\ldots$, the recursion in Definition \ref{formal} defines a map $\colln$ with domain $\Omega$ and range   $\range$:
\begin{eqnarray}
    \colln \! : \Omega & \to & \range  \nonumber  \\
                   \omega & \mapsto &  ( B(0,n)(\omega), B(1,n)(\omega), B(2,n)(\omega),\ldots ).\label{process coupling}
\end{eqnarray}

When $\Omega$ is extended to $(\Omega,\mathcal{F},\p)$ so that $\omega=\t$ is distributed as the sequence of arrival times in the standard Poisson process, the resulting sequence $(B(0,n), B(1,n), B(2,n),\ldots )$ is distributed
as the sequence for the classical occupancy model, with $B(c,n)$ being the number of balls that are needed to get $c$ collisions, for $c=0,1,2,\ldots$, as defined in section \ref{occ.def}.
\end{theorem}
\begin{proof}
To see that $\colln$ maps $\Omega$ \emph{into} $\range$, we argue by induction on $c$.   
In \eqref{recurse},   $\omega \in \Omega$ guarantees that $t_c-t_{c-1} \in (0,\infty)$, hence any $j$ in the set on the right side of \eqref{recurse} is strictly positive.   This set is nonempty, since  $f(n/n) = \infty$, 
hence the infimum of the set is a positive integer at most $n$.  Thus, for every $\omega \in \Omega$, for  $c=1,2,\ldots$, $b_c -b_{c-1}$ is a positive integer, so $\colln(\omega) \in \range$.  

The distributional claim is proved as follows.  In the classical occupancy model, we defined $C(b,n)$, the number of collisions after $b$ balls have been tossed into $n$ bins, via \eqref{notation C}.
  In terms of the map
$\colln$, here we \emph{define} $C(b,n)$ by duality:  $C(b,n) := \max \{ c: B(c,n) \le b \}$. Note that $J(c,n)$ from Definition \ref{formal} corresponds to $I(C(B(c,n),n),n)$, the number of occupied bins after tossing the ball that causes the $c$-th collision.  

For $b=0,1,2,\ldots$, define 
the statement $S(b)$ to be (The joint distribution of $(C(0,n),C(1,n),\ldots,C(b,n))$, resulting from the map
$\colln$ applied to $(\Omega,\mathcal{F},\p)$, is identical to the joint distribution it would have in the classical occupancy model with $n$ bins.)   In proving $S(b)$ for all $b$, the  last sentence of Section
\ref{sect informal} provides the justification for $S(b)$ implies $S(b+1)$. 
\end{proof}

\subsection{Preliminary analysis of the coupling}
We have 
\begin{equation} \label{Tc.recurse}
    T_{c+1} = T_c + \sum_{i \in [I(B(c,n),n),I(B(c+1,n),n \, ))} f(i/n) \ \ + R(c+1)
\end{equation}
where $R$, the random time involved in the last hit step, is limited by the fraction of bins which were occupied just before the $(c+1)$-st collision:
$$ 
  0 <  R(c+1) \le f\left( \frac{I(B(c+1,n)-1,n)}{n} \right) .
$$

Define an auxiliary function
$$   a(i,n) := \sum_{0 \le j < i} f(j/n).$$
Accumulating the hit or miss steps until $B(c,n)$ balls have been tossed, with $c$ hits --- equivalently, unwinding the recursion in \eqref{Tc.recurse} --- gives
$$   T_c = a(B(c,n)-c,n) + \sum_{x=1}^c R(x).$$
Write $b=B(c,n)$ so that $$i = I(B(c,n),n) = b-c $$ is the number of occupied bins when the $c$-th collision is observed,  and use $ f(i/n)$ as an upper bound on $R(1),\ldots,R(c)$.  This yields,  for all $c,n \ge 1$,
\begin{equation}\label{star}
     a(i,n) \  \le T_c \le \ a(i,n) + c \, f(i/n), \quad \text{where $b = B(c,n)$ and  $i = b-c  = I(b,n)$.}
\end{equation}

The contribution from the first-order term of $f(p)$  to $a(i,n)$ is
\[ 
   a_1(i,n) := \sum_{0 \le j < i} j/n = \frac{i(i-1)}{2n} \ge \frac{(i-1)^2}{2n},
\] 
and 
\begin{equation}\label{a1 approx}
   \left| a_1(i,n) - \frac{i^2}{2n}  \right| \le \frac{i}{2n}.
\end{equation}   
   
If $i<n/2$, then for $j <i$, $p := j/n \in [0,1/2)$ has $f(p) \le p + p^2$.    Hence
\begin{equation}\label{f bound}
  \text{ for } i < n/2, \  f(i/n) \le 2i/n
  \end{equation}
and, for $i<n/2$,
\begin{equation}\label{sandwich bounds}
   0 \le a(i,n)-a_1(i/n) = \sum_{0 \le j <i} (f(j/n) - j/n) \le \sum_{0 \le j < i} (j/n)^2 \le i^3/(3n^2).
\end{equation}
A combination of \eqref{star} with \eqref{a1 approx}, \eqref{f bound}, and \eqref{sandwich bounds} is
\begin{equation}\label{combined bound}
   \text{if }  i := B(c,n)-c 
   <n/2,  \\ \text{ then }   \left| \frac{(B(c,n)-c)^2}{2n} - T_c  \right|  \le \frac{i}{2n} + \frac{i^3}{3n^2} + \frac{2ci}{n}.
\end{equation}
\subsection{Theorems giving distributional and sure convergence}

To elucidate the structure of the embedding, we note that the basic convergence, given by \eqref{process converge}, holds for all $\omega$;  it requires only knowing, for each $c=1,2,\ldots$,  
that $T_c(\omega) \in (0,\infty)$.   It  is a standard pattern in probability theory, that distributional convergence for an infinite-dimensional process, as in \eqref{process converge dist display}, is equivalent to the convergence of the finite dimensional distributions, akin to the distributional version of \eqref{process converge}.  But it is remarkable  that even \emph{one}-dimensional convergence, as in \eqref{marginal converge c}, implies joint convergence  \eqref{process converge} and \eqref{process converge dist display}, with a dependent process limit  --- the reason is that  we are dealing with an embedding,  and can argue that there are no exceptional values $\omega$;  alternately, we could have argued about a.s.~convergence,  and noted that, with a  discrete time setup, the countable union of null sets is again a null set. 

\begin{theorem}\label{thm rateless}
Under the coupling given by Theorem \ref{thm embed},
\begin{equation}\label{process converge}
\left(\frac{B(1,n)}{\sqrt{2n}},\frac{B(2,n)}{\sqrt{2n}},\ldots \right) \to (\sqrt{T_1},\sqrt{T_2},\ldots)
\end{equation}
for all $\omega \in \Omega$, as $n \to \infty$.
\end{theorem}

\begin{proof}
To prove \eqref{process converge}, we note first that the usual topology on $\mathbb{R}^\mathbb{N}$ is the compact-open topology, so convergence is equivalent to having, for each fixed $c$, convergence under  the projection into $\mathbb{R}^c$ using the first $c$ coordinates.  Thus, we prove that for fixed $c$, for every $\omega \in \Omega$, as $n \to \infty$,
\begin{equation}\label{process converge c}
\left(\frac{B(1,n)}{\sqrt{2n}},\frac{B(2,n)}{\sqrt{2n}},\ldots,\frac{B(c,n)}{\sqrt{2n}} \right) \to (\sqrt{T_1},\sqrt{T_2},\ldots,\sqrt{T_c}).
\end{equation}

Write $i = B(c,n)-c$;  this is random, varying with $\omega$.
Using the first half of \eqref{star}, $a_1(i,n) \le a(i,n) \le T_c$ so
$$
   (i-1)^2  \le 2n T_c.
$$
Hence $i = B(c,n)-c = O(\sqrt{n})$ as $n \to \infty, $  for every $\omega \in \Omega$ --- the implicit constants in the big Oh depend on $T_c(\omega)$.  For sufficiently large $n$  (again, depending on $\omega$), $i < n/2$, so the upper
bound in \eqref{combined bound}  applies, and 
\begin{equation}\label{big oh}
    \left| \frac{(B(c,n)-c)^2}{2n} - T_c  \right| \le       \frac{i}{2n} + \frac{i^3}{3n^2} + \frac{2ci}{n}   = O(n^{-1/2}),
\end{equation}
using  $i = O(\sqrt{n})$ and $c=O(1)$ to get the final conclusion in \eqref{big oh}.
 Since $T_c(\omega)>0$, \eqref{big oh} implies
  $(B(c,n)-c) \sim \sqrt{2n T_c}$.   Since $c$ is fixed, this implies
\begin{equation}\label{marginal converge c}
\frac{B(c,n)}{\sqrt{2n}}  \to \sqrt{T_c}.
\end{equation}
Finally, \eqref{marginal converge c} implies \eqref{process converge c};  we could have shortened the proof, since
\eqref{marginal converge c} also implies \eqref{process converge} directly, but as discussed before stating this theorem, we want to highlight the unusual nature of the implication:  one-dimensional convergence implies convergence of  the infinite-dimensional joint distributions.
\end{proof}

As an immediate corollary to  the second statement of Theorem \ref{thm embed}, combined with  Theorem \ref{thm rateless}, we get process distributional convergence, as stated formally by Corollary \ref{process converge distribution}.
\begin{cor}\label{process converge distribution}
In the classical occupancy problem, tossing balls into $n$ equally likely bins, as specified in section \ref{occ.def},
as $n \to \infty$, 
\begin{equation}\label{process converge dist display}
\left(\frac{B(1,n)}{\sqrt{2n}},\frac{B(2,n)}{\sqrt{2n}},\ldots \right) \Rightarrow (\sqrt{T_1},\sqrt{T_2},\ldots),
\end{equation}
where $T_c$ is the time of the $c$-th arrival in a standard Poisson process.$\hfill\qed$
\end{cor}

\begin{rmk*}
 The joint distributional limit \eqref{process converge dist display} in Corollary \ref{process converge distribution} of course gives the limit
distribution under 
arbitrary continuous functionals on $\mathbb{R}^\mathbb{N}$, and there are many natural examples where the
scaling by $\sqrt{2n}$ can be removed; for example, as $n \to \infty$
$$
  \p( B(3,n)-B(2,n) > B(1,n)) \to \p( \sqrt{T_3}-\sqrt{T_2}>\sqrt{T_1})
$$
and
$$
  \p( B(1,n)B(5,n) > B^2(3,n)) \to \p( \sqrt{T_1T_5}>T_3).
$$
\end{rmk*}

A result even weaker than Corollary \ref{process converge distribution} answers the basic question for collisions:  What is the approximate distribution of the number $B(c,n)$ of balls that need to be tossed, to get $c$ collisions, when there are $n$ equally likely bins?
\begin{cor}\label{marginal converge distribution}
In the classical occupancy problem, tossing balls into $n$ equally likely bins, as specified in section \ref{occ.def},
for each fixed $c=1,2,\ldots$, as $n \to \infty$, 
\begin{equation}\label{converge dist display}
   \frac{B(c,n)}{\sqrt{n}} \Rightarrow \sqrt{2T_c},
\end{equation}
where $T_c$ is the time of the $c$-th arrival in a standard Poisson process.  
\end{cor}

The distribution of $2 T_c$ is identical to the distribution of the sum of the squares of $2c$ standard normal random variables, and is well known as the Chi-squared distribution with $2c$ degrees of freedom, or the Gamma distribution with shape parameter $c$ and scale parameter 2.  The distribution of $\sqrt{2 T_c}$ is 
known as the Chi distribution with $2c$ degrees of freedom.  This distribution, although not as famous as the Chi-squared distribution, appears naturally also in the tridiagonalization of random symmetric matrices, see \cite[p.~79]{Trotter}.  It can be viewed as a generalized gamma distribution as in \cite[\S17.8.7]{dist}.

\begin{rmk} \label{AGG}
One could count collisions in an alternative way: 
the number of collisions is the sum, over $1 \le i < j \le b$, of the indicator that balls $i$ and $j$ land in the same bin, with the overall effect that a bin containing $k$ balls contributes $\binom{k}{2}$ to the total number of collisions, and the expected number of collisions is $\binom{b}{2}/n$.   This method of counting lends itself to Poisson approximation; see for example \cite[p.~408]{AGG},  and as long as $b=o(n^{2/3})$, the difference between the two method of counting may be considered as an error term, leading to an alternate proof of Corollary \ref{process converge distribution}.
\end{rmk}

\subsection{Theorem giving almost sure asymptotics}

The next theorem \emph{should} be paraphrased as ``Almost surely, if $c \to \infty$ with $c=o(n)$, then $B(c,n) \sim \sqrt{2 c n}$.''  The slightly sloppy paraphrase,  ``If $c \to \infty$ with $c=o(n)$, then $B(c,n) \sim \sqrt{2 c n}$ a.s.''
is a weaker statement, since different sequences $c_1,c_2,\ldots$ might have different null sets, and it is not easy to name a countable collection of sequences which cover all the sequences having $c_n \to \infty$ and $c_n/n \to 0$. 

\begin{theorem}\label{thm asymptotic}
Under the coupling given by Theorem \ref{thm embed},
the good event \mbox{$G = \{ \lim_{c\to\infty} T_c/c = 1 \}$} has probability $1$. For all $\omega \in G$, 
for any sequence $c_1,c_2,\dots$ of positive integers such that  $c_n \to \infty$ and $c_n/n \to 0$, we have
\begin{equation}\label{almost sure asymptotic}
    \frac{B(c_n,n)}{\sqrt{2 c_n n}}  \to  1.
\end{equation}
\end{theorem}
\begin{proof}
 $\p(G)=1$ by the strong law of large numbers. 
Write $c$ for $c_n$ and $i=(B(c,n)-c)$.   For $\omega \in G$, the relation   
$$
   (i-1)^2  \le 2n T_c
$$
with $T_c \sim c = o(n)$ implies that $i^2 = O(nc) = o(n^2)$, hence $i=o(n)$.   Hence for sufficiently large $n$ (depending on the choice of $\omega \in G$),
\eqref{combined bound} applies, giving:
\begin{equation}\label{little oh}
    \left| \frac{(B(c,n)-c)^2}{2n} - T_c  \right| \le       \frac{i}{2n} + \frac{i^3}{3n^2} + \frac{2ci}{n}   = o(c).
\end{equation}
The equality 
is justified term-by-term, where the argument for the middle  term is that $i^3/n^2 = (i^2/(nc) ) \ (i/n) \ c = O(1) \
o(1)\ c = o(c)$.  Using $\omega \in G$, \eqref{little oh} implies that $i^2/(2n) \sim c$, equivalently
$i^2 \sim 2nc$, equivalently $B(c,n)-c \sim \sqrt{2nc}$.   Finally, since $ c=o(\sqrt{2nc})$, this implies that
$B(c,n) \sim \sqrt{2nc}$.
\end{proof}

\subsection{Theorem giving almost sure uniform convergence}

\begin{theorem}\label{thm uniform}
Under the coupling given by Theorem \ref{thm embed},
the good event $H = \{ T_c \asymp c \}$, that the $T_c$-to-$c$ ratio is bounded away from zero and infinity, 
has probability $1$.   For all $\omega \in H$, 
there is uniform convergence, as given by the following:
\begin{equation}\label{1/3 converge}  
  \sup_{c=o(n^{1/3})} \left| \frac{B^2(c,n)}{2n} - T_c \right|  \ \to 0; 
\end{equation}
\begin{equation}\label{1/2 converge}   
  \sup_{c=o(n^{1/2}) } \left| \frac{B(c,n)}{\sqrt{2n}} - \sqrt{T_c} \right|  \ \to 0;  
\end{equation}
for any $C < \infty$ and $\alpha \in (0,1/3)$,
\begin{equation}\label{alpha 1/3 converge}  
  \sup_{c \le Cn^\alpha} \left| \frac{B^2(c,n)}{2n} - T_c \right|  = O( n^{(3\alpha-1)/2});
\end{equation}
and
for any $C < \infty$ and $\alpha \in (0,1)$,
\begin{equation}\label{alpha  bound}  
 \  \sup_{c \le Cn^\alpha} \left| \frac{B(c,n)}{\sqrt{2n}} - \sqrt{T_c} \right|  = O( n^{\alpha-(1/2)}).  
\end{equation}

\end{theorem}

\begin{proof}
Observe that in \eqref{combined bound}, each ingredient in the upper bound, $i=B(c,n)(\omega)-c= J(c,n)(\omega)$, and $c$ itself, is a nondecreasing function of $c$.  So immediately, we also have the stronger uniform statement
$   \text{if }  i := B(c,n)-c \text{ satisfies } i<n/2,   \text{ then }   $
\begin{equation}\label{combined bound uniform}
  \sup_{1 \le c' \le c} \left| \frac{(B(c',n)-c')^2}{2n} - T_{c'}  \right|  \le \frac{i}{2n} + \frac{i^3}{3n^2} + \frac{2ci}{n}.
\end{equation}

The exact meaning of \eqref{1/3 converge} is:  
for any sequence $c_1,c_2,\ldots$, such that $c_n/n^{1/3} \to 0$,
\begin{equation}\label{1/3 converge precise}  
\text{ for all $\omega \in H,$} \quad \lim_{n\to\infty} \sup_{1 \le c' \le c_n} \left| \frac{B^2(c',n)}{2n} - T_c' \right|  = 0 .
\end{equation}
We may assume that $c_n \to \infty$, for if $\sup c_n < \infty$, then \eqref{1/3 converge precise} holds simply as a corollary of \eqref{marginal converge c} in Theorem \ref{thm rateless}.

Write $c=c_n$ and $i = B(c,n)-c$.
Using the first half of \eqref{star}, $a_1(i,n) \le a(i,n) \le T_c$ so
$$
   (i-1)^2  \le 2n T_c,
$$
hence $i =  O(\sqrt{nc}) =o(n)$ as $n \to \infty, $  for every $\omega \in H$.  For sufficiently large $n$  (depending on $\omega$), $i < n/2$, so the upper
bound in \eqref{combined bound uniform}  applies, and 
\begin{equation}\label{big oh uniform}
\left| \frac{(B(c',n)-c')^2}{2n} - T_{c'}  \right| \le       \frac{i}{2n} + \frac{i^3}{3n^2} + \frac{2ci}{n}  =: u  = O(\sqrt{c^3/n}) = o(1),
\end{equation}
   for $1 \le c' \le c$,
noting that each of  $i^3/n^2$ and  $ci/n$  just satisfies the $O(\sqrt{c^3/n})$ relation.   

Once $n$ is large enough that the upper bound $u$ in \eqref{big oh uniform} is less than 1, and large enough so that $T_c \le Uc$ for some constant $U$, we have, for all $1 \le c' \le c=c_n$,
$$
 \frac{(B(c',n)-c')^2}{2n} \le  T_c + 1 \le Uc + 1 =: t,
$$ 
which implies that $(B(c',n)-c')^2 \le 2nt$, hence $B(c',n) \le \sqrt{2nt} + c' \le \sqrt{2nt}+c$, hence expanding $(B(c',n)-c')^2$ and using the triangle inequality,
\begin{eqnarray}
\left| \frac{(B(c',n))^2}{2n} - T_{c'}  \right| & \le & \frac{2c' B(c',n)+c'^2}{2n} +  \left| \frac{(B(c',n)-c')^2}{2n} - T_{c'}  \right| \nonumber \\  
   & \le &   \frac{2c( \sqrt{2nt}+c)+c^2}{2n}+  u   = O( \sqrt{c^3/n}) \label{big oh uniform 2}
\end{eqnarray}
   for $1 \le c' \le c$.
This completes the proof of \eqref{1/3 converge}, and simultaneously proves \eqref{alpha 1/3 converge}.
We note  that \eqref{K big oh 2} below will provide an alternate proof of \eqref{1/3 converge} and \eqref{alpha 1/3 converge}, without making use of the uniformity in \eqref{combined bound uniform}.

Next, we prove \eqref{1/2 converge}.
Consider the random variables $L :=\inf_{c \ge 1} T_c/c, \ U := \sup_{c \ge 1} T_c/c$, so  that $\omega \in H$ is precisely that $0 < L(\omega) \le U(\omega)  < \infty$.  For every $c$, writing $i=B(c,n)-c$, the first half of \eqref{star}, that $a_1(i,n) \le a(i,n) \le T_{c}$, yields $   (i-1)^2  \le 2n T_{c},$
so $\omega \in H$ yields the further $(i-1)^2 \le 2 n U c   < \infty$, 
hence there is a random $K(\omega) < \infty$ such that for all $n,c \ge 1$, $i := B(c,n)-c \le K \sqrt{nc}$.   Note that this ``big Oh" conclusion is weaker than the asymptotic in \eqref{almost sure asymptotic},  but stronger in that it requires no condition on the growth of $c$ relative to $n$.
Next, for any $\varepsilon > 0$,  imposing the  growth condition  $c \le n^{1-\varepsilon}$,  the relation $i := B(c,n)-c \le K \sqrt{nc}$ implies that there exists $n_0(\omega)< \infty$ such that for all $n > n_0$, for all $c \le n^{1-\varepsilon} $, $i := B(c,n)-c \le n/2$, enabling \eqref{combined bound}, which we further bound using $i \le K \sqrt{nc}$:
\begin{equation}\label{K big oh}
    \left| \frac{(B(c,n)-c)^2}{2n} - T_c  \right| \le       \frac{i}{2n} + \frac{i^3}{3n^2} + \frac{2ci}{n}  \le 2K^3(\sqrt{c/n}+\sqrt{c^3/n})
    \le 4K^3 \sqrt{c^3/n} .
\end{equation}
Next, similar to \eqref{big oh uniform 2}, we expand $(B(c,n)-c)^2$ and apply the triangle inequality;  further using $B(c,n) \le c+K \sqrt{nc}$  hence $2 c B(c,n)+c^2  \le 3c^2 + 2Kc \sqrt{nc}$ yields:  for $\omega \in H$, for $n \ge n_0(\omega)$, for all $c \le n^{1-\varepsilon}$,
\begin{equation}\label{K big oh 2}
    \left| \frac{(B(c,n))^2}{2n} - T_c  \right|   \le 4K^3 \sqrt{c^3/n} + \frac{3c^2 + 2Kc \sqrt{nc}}{2n} \le \frac{3c^2}{2n}+ K' \sqrt{c^3/n} = O(\sqrt{c^3/n}),
\end{equation}
and we have made the random implied constant in the big Oh more or less explicit.

What is the effect of applying square root, to both $B^2(c,n)/(2n)$ and to $T_c$, in \eqref{K big oh 2}?  Suppose that $c=n^\alpha$ with $\alpha \in (0,1)$;  the  random $T_c$ is of order $T_c \asymp c = n^\alpha$,  and suppose  the amount of perturbation, $d := B^2(c,n)/(2n) - T_c$, is of order $d \asymp \sqrt{c^3/n} = n^{(3\alpha-1)/2}$.   This situation has $d/c = n^{(\alpha-1)/2}$, which is $o(1)$ provided $\alpha<1$.   Using $d=o(c)$, we have
$\sqrt{c+d}= \sqrt{c} \sqrt{1+d/c}  = \sqrt{c} (1 + d/(2c) + O(d^2 / c^2))$, leading to $\sqrt{c+d}-\sqrt{c} \sim \sqrt{c} \, d/(2c) = d/(2\sqrt{c}) \asymp d/\sqrt{c} \asymp n^{(3\alpha-1)/2}/n^{\alpha/2} = n^{\alpha - \frac{1}2}$.   The combination of this calculation, with the uniform upper bound \eqref{K big oh 2}, together with $\omega \in H$, proves both \eqref{1/2 converge} and \eqref{alpha bound}.  
 \end{proof}

\section{Uniform integrability in the uncentered case, $c = O(n)$}\label{sect ui}

\subsection{Uniform integrability}
Lemma \ref{crucial} gives the crucial estimate, and Proposition
\ref{UI} establishes uniform integrability, as required in the proof of Corollary \ref{cor fixed c}, to get limit moments from the distributional convergence proved in 
Corollary \ref{marginal converge distribution}.

\begin{lem} \label{crucial}
Fix $0 < K < \infty$.  For $n=1,2,\ldots,$ for all $ c$ with $ 1 \le c \le Kn$, for all $t>\max(8K,44)$,
\[
\P(B(c,n)/\sqrt{2cn}> \sqrt{t}) \le ce^{-ct/8}.
\]
\end{lem}

\begin{proof}
The  restriction $t \ge 44$ implies that $t/8 \ge \log(2t)+1$.  The restriction $1 \le c \le Kn$ implies $2(c+1)^2/(nc)  \le 2(2c)^2/(nc) = 8c/n \le 8K$, so combined with $t \ge 8K$  we have $2(c+1)^2/(nc)  \le t$, so $(c+1)^2 \le ntc/2$.

Put $b$ for the floor of $\sqrt{2nct}$, so:
\[
\P(B(c,n)  /\sqrt{2cn}> \sqrt{t}) = \P(B(c,n) > b).
\]
Using \eqref{define B}, this equals $\P(C(b,n) \le c-1)$, where $C(b,n)$ is the random variable denoting the number of collisions obtained after throwing $b$ balls.  The event $C(b,n) = y$ entails partitioning the set of $b$ balls into $b - y$ blocks, i.e., disjoint nonempty subsets.  For $y = 0, 1, \ldots, c-1$, we have
\[
\P(C(b,n) = y) \le \binom{b}{y} b^y \frac{(n)_{b-y}}{n^b}
\]
where the binomial coefficient is for choosing which $y$ of the $b$ balls were \emph{colliders} (i.e., landed in a non-empty bin), the $b^y$ term overcounts which ball each of the colliders collided with, $(n)_{b-y}$ describes the assignment of the $b-y$ blocks to bins, and there are $n^b$ possible throws.

We substitute $\binom{b}{y} \le b^y/y!$ and $b^2 \le nt^2$ to obtain:
\[
\P(C(b,n) = y) \le \frac{t^{2y}}{y!} \frac{(n)_{b-y}}{n^{b-y}}.
\]
The second fraction is bounded above by $\exp(-\binom{b-y}2/n)$.  Note that by our hypothesis on $t$, $c+1 < t\sqrt{n/2}$, hence
\[
b-y-1 = (b+1) - (y+1)-1 \ge t\sqrt{n} - c - 1 \ge t\sqrt{n}/2
\]
and $\binom{b-y}2 \ge nt^2/8$.  Omitting the $1/y!$ term gives
\[
\P(C(b,n) = y) \le t^{2y} \exp(-nt^2/8) \le \exp(-(t^2/8 - 2y\log(t))).
\]
As $t^2/8 - 2y\log(t) \ge t^2/8 - 2c\log(t) \ge t^2/16$, it follows that $\P(C(b,n) < c) \le ce^{-t^2/16}$, proving the claim.
\end{proof}

\begin{lem}\label{UI}
Fix $K < \infty$.  For $k=1,2,\ldots$, the family 
\[
\left\{ \left(\frac{B(c, n)}{\sqrt{cn}} \right)^k : n \ge 1, c \le Kn \right\}
\] 
is uniformly integrable.
\end{lem}

\begin{proof}
The exponentially decaying uniform upper bound on the upper tail, given by Lemma \ref{crucial}, implies that  for each $k=1,2,\ldots,    \sup_n \sup_{ c \le Kn}  \e (B(c,n)/\sqrt{cn})^{k+1} < \infty$.  This uniform boundedness of the
$(k+1)$-st moments implies uniform integrability of the family  of $k$-th powers.
\end{proof}

\section{Moments in the uncentered case, $c = o(n)$}

\subsection{Corollaries of convergence together with uniform integrability}

\begin{cor}\label{cor fixed c}
For each fixed $c$, and for $k=1,2,\ldots$
\begin{equation}\label{fixed c moments}
   \e [B(c,n)^k]  \sim (2n)^{k/2} \, \e [T_c^{k/2}]     \text{ as } n \to \infty ,
\end{equation} 
where $T_c$ is the time of  the $c$-th arrival in the standard rate 1  Poisson process.
\end{cor}

\begin{proof}
Combine the  one-dimensional distribution convergence in Corollary \ref{marginal converge distribution} with the uniform integrability in
Lemma \ref{UI}.
\end{proof}

\begin{cor}\label{cor c grows}
For $c=c_n$ with $c_n \to \infty$ and $c_n/n \to 0$, and for $k=1,2,\ldots$, 
\[
\e [B(c_n,n)^k] \sim (2 n c_n)^{k/2}   \text{ as } n \to \infty
.
\]
\end{cor}

\begin{proof}
Combine the almost sure limit in Theorem \ref{thm asymptotic} with the uniform integrability in
Lemma \ref{UI}.
\end{proof}

\subsection{Explicit asymptotics for the mean and variance of $B(c,n)$,  fixed $c$.}\label{sect gamma}
Define, for $c=1,2,\ldots$,
$$
   \gamma(c) := \frac{\e [T_c^{1/2}]}{\sqrt{c}},
$$
so that \eqref{fixed c moments}, specialized to $k=1$, may be paraphrased as
\begin{equation}\label{first asymptotic moment}
\e B(c,n) \sim \gamma(c)  \sqrt{2nc} \quad \text{for fixed $c=1,2,\ldots$}
\end{equation}
Recalling that $T_c$, the time of the $c$-th arrival in a standard Poisson process, has density  $f_c(x) = x^{c-1} e^{-x}/\Gamma(c)$ for $x >0$, we calculate explicitly   
\begin{equation}\label{gamma formula 0}
\gamma(c) = \Gamma(c + 1/2) /(\sqrt{c} \Gamma(c)).
\end{equation}
It is an exercise using the material in \cite[Chap.~1]{AAR} to show that $\gamma(c)$ is increasing.

Calculating explicitly $\gamma(c)$ for a few values of $c$ gives:
\[
\begin{array}{c|ccccccc} 
c&1&2&3&4&5&\cdots&\infty \\ \hline
\gamma(c) & \sqrt{\frac{\pi}4} & \sqrt{\frac{9\pi}{32}} & \sqrt{\frac{75\pi}{256}} & \sqrt{\frac{1225\pi}{4096}}&\sqrt{\frac{19845\pi}{65536}}&\cdots&1
\end{array}
\]
where the limit $\gamma(c) \to 1$ is the $k=1$ case of the well-known formula \[\lim_{c\to \infty} \frac{\Gamma\left(c + \frac{k}2\right)}{\Gamma(c) \, \sqrt{c}^k} = 1.\]
So $\gamma(c)$ increases from about $0.886$ to 1.  And $c$ need not be very big for $\gamma(c)$ to be close to 1; for $c \ge 13$, $\gamma(c) > 0.99$.  

Plugging numerators into Sloane's \cite{OEIS} gives a hit on 
sequence A161736, ``Denominators of the column sums of the BG2 matrix",
giving the formula 
\begin{equation}\label{gamma formula}
\gamma(c) = \frac{(2c)!}{2^{2c}(c!)^2} \sqrt{\pi c},
\end{equation}
whose verification from \eqref{gamma formula 0} is easy, by induction on $c$, starting from $\Gamma(3/2)= \sqrt{\pi}/2$.
Now Stirling's approximation can be applied to show how $\gamma(c)$ approaches 1 as $c$ grows:
\begin{equation} \label{cx}
\gamma(c) = 1 - \frac{1}{8c} + \frac{1}{128c^2} + \frac{5}{1024c^3} - \frac{21}{32768c^4} + O(1/c^5).
\end{equation}

For fixed $c$,  Corollary \ref{cor fixed c} also
gives the asymptotic variance.
\begin{cor} \label{var}
For constant $c$,
\begin{equation}\label{gamma formula var}
\lim_{n \to \infty}   \frac{\Var B(c,n)}{n} =  2c(1 - \gamma(c)^2).
\end{equation}
\end{cor}
\begin{proof}
The implication is immediate, by taking $k=1$ and $k=2$ in \eqref{fixed c moments}.
\end{proof}  
Calculating explicitly the expression  in \eqref{gamma formula var} for the first few values of $c$ gives

\[
\begin{array}{c|ccccccc} 
c&1&2&3&4&5&\cdots&\infty \\ \hline
\displaystyle\lim_{n\to\infty} \frac{\Var B(c,n)}{n} 
& 0.4292 
& 0.4567 
& 0.4777 
& 0.4835 
& 0.4869 
&\cdots&0.5
\end{array}
\]
and 
the 
limit of $(\Var B(c,n))/n$ has the series expansion
\[
\lim_{n\to \infty}\frac{\Var B(c,n)}n = \frac12 - \frac1{16c} - \frac1{64c^2} + \frac5{1024c^3} + \frac{23}{4096c^4} + 
\cdots .
\]
The last column of the table above means only that
\begin{equation} \label{var.limit}
  \lim_{c \to \infty}   \lim_{n\to \infty}\frac{\Var B(c,n)}n = \frac12.
\end{equation}

\section{Results for collisions, based on duality}\label{sect dual}
In the remainder of the paper, we find the asymptotic variance of $B(c,n)$ when $c=c_n \to \infty$ with $c/n \to \alpha_0 \in [0,\infty)$.
Our method is to combine duality  with R\'enyi's central limit result for the number of empty bins, to get a normal limit for $B(c,n)$ (this section), and to prove a concentration result to get uniform integrability (section \ref{UI.duality}), so that the normal limit governs the asymptotic variance (see section \ref{moments.duality}).

\subsection{History:  Weiss and R\'enyi}
Weiss in 1958, \cite{Weiss}, proved a central limit theorem for $N_0(b,n)$ in the ``central regime'', where 
$b,n \to \infty$ with $b \asymp n$, i.e., with ratio
bounded away from zero and infinity.  Weiss explicitly stated that $N_0$ is asymptotically normal, and implicit in this, together with his proof, is that the interpretation of asymptotic normality involves subtracting off the mean of $N_0$ and dividing by the standard deviation of $N_0$, i.e., Weiss proved that
$$
    \frac{N_0 - \e N_0}{\sqrt{\var N_0}} \Rightarrow Z.
$$ 
Renyi in 1962, \cite{Renyi:orig}, gave 3 proofs and went a little further.  He gave a nice explicit expression to approximate the mean and variance, and, in his Theorem 2, ``the third proof,'' extended the result to the case $b=o(n), b^2/n \to \infty$.  

For motivation, suppose that $b/n \to \lambda$, so that a fixed number  $\lambda $ serves as the limit average number of balls per bin, $e^{-\lambda}$ is the limit probability that a  given bin is empty, and the number of empty bins is asymptotic to $n e^{-\lambda}$.  
With 
\begin{equation}\label{def d}
  d(x) :=  e^{-x} \left( 1- (1+ \ x ) e^{-x} \right) \eand
\sigma^2(b,n) :=  d(b/n),
\end{equation}
Renyi  observes that $\var N_0(b,n) =n \,  \sigma^2(b,n) (1+O(b/n^2))$ --- so $ n \, \sigma^2$ is not the variance,
but rather, a nice approximation of the variance.  Likewise, $n e^{-b/n}$ 
is not $\e N_0$, but rather, a nice approximation.  We remark that for the case $x \to 0$,
\begin{align} \label{d.mac}
d(x) &= e^{-x}(1-(1+x)e^{-x}) \sim (1-(1+x)e^{-x}) \\ 
& = 1-(1+x)(1-x+x^2/2-O(x^3)) = x^2/2 + O(x^3) \sim x^2/2. \notag
\end{align}

A restatement of R\'enyi's Theorem 2, using our notation, is: \emph{If $b,n \to \infty$  with  $b=O(n)$ and  $b^2/n \to \infty$, then}
\begin{equation}\label{renyi}
    \frac{N_0(b,n) - n e^{-b/n}}{\sqrt{n} \, \sigma(b,n)} \Rightarrow Z.
\end{equation}
More advanced versions of Renyi's theorem, with concrete error bounds, are given by \cite{Englund,Mikhailov,BG}, 
but for our purpose, to get a central limit for the number of collisions, using duality as in the proof of Theorem \ref{MT}, R\'enyi's \eqref{renyi} is ideal.

\subsection{Unified normal limit, using duality}

The treatment in this section
is \emph{unified} in the sense that 
it handles both the regime $c_n \to \infty$ with $c_n=o(n)$, where the number of balls per bin approaches zero, and the regime
$c_n \sim \alpha_0 \, n$ with $\alpha_0 \in (0,\infty)$, where the number of balls per bin approaches a limit $\lambda_0 \in (0,\infty)$.

We define the  function $w \!: [0, \infty) \to [0, \infty)$ via 
\begin{equation}\label{def w}
w(x) := e^{-x} + x -1.
\end{equation}
  It is easily checked that $w$ is strictly increasing and onto, with $w(0) = 0$.  Given $c, n > 0$, we define
\begin{equation}\label{def beta}
\beta(c, n) := nw^{-1}(c/n).
\end{equation}
The function $w$ has the Maclaurin series:
\[
w(x) =  \frac{x^2}2 - \frac{x^3}6 + \frac{x^4}{24} - \cdots \quad \text{for small $x \ge 0$.}
\]
Therefore the inverse function has the expansion
\begin{equation} \label{winv}
w^{-1}(\delta) = \sqrt{2\delta} + \frac{\delta}3 + \frac{\delta^{3/2}}{9\sqrt{2}} + \frac{2\delta^2}{135} - \cdots \quad \text{for small $\delta\ge 0$.}
\end{equation}
In particular, when $c_n$ is a function of $n$ so that $c_n/n \to 0$, 
we have
\[
\beta(c_n,n) = \sqrt{2c_nn} + \frac{c_n}3 + \frac{1}{9\sqrt{2}} \frac{c_n^{3/2}}{\sqrt{n}} + \frac2{135} \frac{c_n^2}{n} - \cdots
\]
and  $\beta(c_n,n) \sim \sqrt{2c_nn}$.

We  define a continuous function $g(x)$ on $[0, \infty)$ via 
\begin{equation}\label{def g}
g(0) := \frac1{\sqrt{2}} \eand g(x) := \frac{\sqrt{d(x)}}{1 - e^{-x}}\ \text{for $x > 0$,}
\end{equation}
where $d(x)$ is as in \eqref{def d}.

\begin{theorem} \label{MT}
Suppose $\lim_{n \to \infty} c_n = \infty$ and $\lim_{n \to \infty} c_n/n = \alpha_0 \in [ 0, \infty)$.  Then, 
with $\beta(c,n)$ defined by \eqref{def beta} to give the centering,
with $w$  as defined by \eqref{def w},    $\la_0 := w^{-1}(\alpha_0)$,  and $g$ as defined by  \eqref{def g} to give the scaling,
we have the following convergence in distribution to the standard normal random variable $Z$:
\begin{equation}\label{our normal limit}
\frac{B(c_n, n) - \beta(c_n, n)}{g(\la_0) \sqrt{n}} \Rightarrow Z.
\end{equation}
\end{theorem}

\begin{proof}

Let $c=c_n$ be given, with $c \to \infty$ and $c/n \to \alpha_0 \in [0,\infty)$.  Write $\beta=\beta(c,n), \lambda=\beta/n$ so that  \eqref{def beta} says
that $c/n=w(\beta/n) $, i.e., $c=n w(\lambda)$, i.e.,
\begin{equation}\label{c relation}
c = n \, e^{-\lambda} + n \lambda -n.
\end{equation}

For fixed real $y$, let
\begin{equation}\label{beta b}
    b = \beta + y \sqrt{n}, \text{ so } b/n  = \lambda + y/\sqrt{n}.
\end{equation}
In terms of the cumulative distribution function $\Phi$ for the standard normal, so that $\p(Z>y)  = \p(Z< -y) =: \Phi(-y)$, \eqref{our normal limit} is equivalent to showing that
for each fixed $y$, $\p(B(c,n)>b) \rightarrow  \p(g(\lambda_0) \, Z>y) = \Phi(-y/g(\lambda_0))$.  

To enable R\'enyi's result \eqref{renyi}, we need to check that $b > 0$ for sufficiently large $n$, that $b/n$ is bounded,  and that $b^2/n \to \infty$.
As $b/n = \beta/n + y/\sqrt{n} = w^{-1}(c_n/n) + y/\sqrt{n}$ and $c_n$ is $O(n)$, $b/n$ is bounded.  Further, $\lim_{n \to \infty} \sqrt{n} w^{-1}(c_n/n) = \infty$ --- this is obvious if $\alpha_0 \ne 0$ and follows from $c_n \to \infty$ and \eqref{winv} if $\alpha_0 = 0$ --- hence $\lim_{n \to \infty} b/\sqrt{n} = \infty$. 

Recall, from \eqref{notation I} and \eqref{notation C}, that  $C(b,n)$,  the number of collisions obtained after throwing $b$ balls, and $N_0(b,n)$ for the number of empty bins remaining after throwing $b$ balls, are related by
\[
C(b,n) = b - (n - N_0(b,n)).
\]
So by duality we have:
\[
  \p(B(c,n)> b)    = \p( C(b,n) < c)  = \p( N_0(b,n)  < n -( n \lambda + y \sqrt{n})  +c).
\]
Applying \eqref{c relation}, we find:
\begin{align}
\p(B(c,n)>b) &= \p( N_0(b,n)  <  n e^{-\lambda} - y \sqrt{n} ) \label{middle big display} \\
&   = \p\left( N_0(b,n) - n e^{-b/n} <  y \sqrt{n}(e^{-\lambda} -1) + O(1)\right). \label{middle.2}
\end{align}

We remark that $\sqrt{n} \, \s(b, n) = \sqrt{n\,d(b/n)} \to \infty$.  Indeed, as $b/n$ is bounded and $d(x)$ is nonzero for positive $x$, it suffices to check this in the case where $b/n \to 0$, where it follows from \eqref{d.mac}.    Therefore, dividing both sides of the inequality in \eqref{middle.2} gives
\[
\p(B(c,n) > b) = \p\left( \frac{N_0(b,n) - n e^{-b/n}}{\sqrt{n}\,\s(b,n)} <  \frac{y(e^{-\lambda} -1)}{\s(b,n)}\right).
\]
By \eqref{renyi}, to complete the proof of the theorem it remains only to verify that
\begin{equation} \label{renyi.verify}
\s(b,n)/(e^{-\la}-1) \to -g(\la_0).
\end{equation}

If $\alpha_0=0$ then  $\lambda_0  = w^{-1}(0) = 0$.  We have $  \beta(c,n) \sim \sqrt{2cn}$ hence $c \to \infty$ implies $b \sim \beta$.
Using $\lambda=\beta/n \to 0$ we get $\s^2(\beta,n) \sim \la^2/2$ by \eqref{d.mac},
so $\sigma(\beta,n) \sim \lambda/\sqrt{2}= \lambda g(\lambda_0)$.  It follows similarly, since $b \sim \beta$, that
$\sigma(b,n) \sim \sigma(\beta,n) \sim \lambda g(\lambda_0)$.  And of course $e^{-\lambda}-1 \sim -\lambda$; \eqref{renyi.verify} follows.

If $\alpha_0 > 0$  then $\lambda \to \lambda_0 = w^{-1}(\alpha_0) > 0$, and $\sigma^2(b,n) \to d(\lambda_0)$, a number, verifying \eqref{renyi.verify}.
This concludes the proof of the theorem.
\end{proof}

\section{Uniform integrability and concentration in the centered case
} \label{UI.duality}

\subsection{Overview:  background on concentration inequalities}

In order to conclude, from Theorem \ref{MT}, that the variance of $B(c,n)$ is asymptotic to $n \, g(\lambda_0)^2 $, we need uniform integrability.  
As in the proof of Theorem \ref{MT}, fluctuations for $B(c,n)$, the number of balls that must be tossed to get $c$ collisions,  are related via duality to fluctuations of 
\begin{equation}\label{C N_0 redundant}
   C(b,n) = b - (n - N_0(b,n)),
\end{equation}
the number  of collisions resulting from tossing $b$ balls.
Hence we investigate concentration bounds for $N_0(b,n)$, or directly equivalently, concentration bounds for $C(b,n)$.  

The central region, with $c$ and $b$ both of order  $n$, is relatively easy to handle.  In contrast, it took much effort to
understand the region of main interest here:  $c \to \infty$ with $c=o(n)$, so that $c=o(b)$ and $b=o(n)$.
The following three random variables with exactly the same variance, but in the region of main interest, their expectations have different order of growth.  From large to small, they are:
\begin{eqnarray*}
     N_0(b,n) & \text{ with } & \e N_0(b,n) \sim n,  \\
     n-N_0(b,n) & \text{ with } & \e (n - N_0(b,n)) \sim b, \text{ and }  \\  
       C(b,n) = b - (n - N_0(b,n))   & \text{ with } & \e C(b,n) \sim c .
\end{eqnarray*}

Consider applying Azuma's  inequality for martingales with bounded differences,  with random variables $X_1,\ldots,X_b$
$ \in \{1,2,\ldots,n \}$  
to say the destination bin for each ball, and sigma-algebras $\cF_i := \sigma(X_1,\ldots,X_i)$ to carry the information known when the first $i$ balls have been tossed, and $M_i := \e(N_0(b,n)|\cF_i)$ for the martingale.  It is obvious that $|M_i-M_{i-1}| \le 1$ for $i=1$ to $b$, so Azuma gives the bounds  
\begin{eqnarray}\label{azuma}
\p(N_0(b,n) - \e N_0(b,n) \ge t) &\le& \exp(-t^2/(2b)),  \\
\ \ \p(N_0(b,n) - \e N_0(b,n) \le -t) &\le& \exp(-t^2/(2b))  \nonumber.
\end{eqnarray} 
For the central region, where $c$ and $b$ are both of order $n$, these bounds give us enough concentration to prove the desired UI result.  In contrast, for the region of main interest to us, with $c=o(n)$ and $b \sim \sqrt{2cn}$, the bounds \eqref{azuma} are inadequate.
For the region of main interest, 
we use a bounded size bias coupling whose existence is provided in the next subsection.

\subsection{Bounded size-biased couplings for the number of collisions} \label{sizebias}
\begin{prop}
There is a coupling of $C(b,n)$ with its size-biased version $C'(b,n)$ such that $C'(b,n) - C(b,n) \in \{-1, 0, 1, 2\}$ for all outcomes.  Furthermore, if each $X_i$ as defined in the previous paragraph is uniformly distributed on the boxes $1, \ldots, n$ (as assumed in the rest of this paper), then there is a coupling such that $C'(b,n) - C(b,n) \in \{ 0, 1 \}$ for all outcomes.
\end{prop}

\begin{proof}
We assume that the $X_1,\ldots,X_b$
are mutually independent,  but not necessarily uniformly distributed, nor even identically distributed.
We will consistently use the following notation:  $1 \le i < j \le b, 1 \le k \le n$, so that $i$ and $j$ refer to balls,
and $k$ to bins,  and $i$ is tossed before $j$.  Note, this entails $j \ge 2$.

Write $Z_{ik} \equiv Z_{i,k} = 1(X_i=k)$ for the indicator that ball $i$ lands in box $k$.  The indicator that ball $j$ lands in box $k$ \emph{and} accounts for a new collision --- because at least one earlier ball had already landed in box $k$ --- is
$$
  Y_{jk} = Z_{jk} \, 1( Z_{1k}+Z_{2k}+\cdots +Z_{j-1,k} > 0)
$$
and the indicator that ball $j$, when it lands, accounts for a new collision, is
$$
   W_j = \sum_{k=1}^n Y_{jk}.
$$
Hence the total number of collisions, when $b$ balls are tossed into $n$ boxes, can be expressed as
\begin{equation} \label{C.as.sum}
   C(b,n) = \sum_{j=2}^b  W_j  \end{equation}
or
\begin{equation} \label{C.as.sum2}
   C(b,n)  = \sum_{j=2}^b \sum_{k=1}^n Y_{jk} 
\end{equation}

We recall some basics about size bias, as presented by \cite[equations (15) and (12)]{AGK}.  First, for sums as 
\eqref{C.as.sum} or \eqref{C.as.sum2}, the size bias distribution is naturally expressed as a mixture, with weights proportional to the contribution that a single term makes to the expected sum, of the sum for the process where the joint distribution of summands is biased in the direction of the chosen summand.  Second, when the summands are indicators,  biasing in the direction of the chosen summand is the same as conditioning on the event indicated by the chosen summand.  Finally,  when the summands, such as those in
\eqref{C.as.sum} or  \eqref{C.as.sum2}, are derived from an underlying process describing where every ball lands, such as $\BX=(X_1,X_2,\ldots,X_b)$,  then biasing the process of summands can be done by conditioning the entire underlying process on the event indicated by the chosen summand.  We will find, for each summand, a coupling of $\BX$  with $\BX'=(\BX$, conditional on the event indicated by that summand), which will give a coupling of the original $C(b,n)$ with the conditioned version $C'(b,n)$ such that $|C'(b,n) - C(b,n)|$ is bounded.

Consider the  sum in \eqref{C.as.sum2}.  Assume $\e Y_{jk}>0$.  The event indicated by $Y_{jk}$ is an intersection of two independent events, so conditioning on this is the same as conditioning on ball $j$ landing in box $k$, and  at least one of balls $1,2,\ldots,j-1$ landing in box $k$.  The sum $S=Z_{1k}+Z_{2k}+\cdots +Z_{j-1,k}$  is a sum of independent Bernoulli random variables, so by the sandwich principle \cite[Cor.~7.4]{AB} there is a coupling of $S$ with $S'$ in which $S'-S \in \{0,1\}$ for all outcomes $\omega$, where $S'$ is distributed as $S$ conditioned to be nonzero.  

This coupling of $S$ with $S'$ lifts to a coupling of $\BX$ with $\BX'$, in which $X_i=X_i'$ for all $i>j$, $X_j'=k$, and, either  $S'-S=0$ and $X_i=X_i'$ for $i=1$ to $j-1$, or else $S'-S=1$  and  $X_i=X_i'$ for $i=1$ to $j-1$ with a single exception $I$, with $X_I\ne k$,
and $X_I'=k$.  (To see this, begin with the observation that we have given values $p_1,\ldots,p_{j-1}$  with $p_i = \e Z_{ik} = \p(X_i =k)$,  and there is a unique distribution for
a permutation $(I_1,\ldots,I_{j-1})$ of $\{1,2,\ldots,j-1\}$,   such that starting from the all zero vector in $\{0,1\}^{j-1}$, and changing coordinates one at a time to one, according to the indices $I_1,\ldots, I_{j-1}$,  yields the process
$(Z_{ik},Z_{2k},\ldots,Z_{j-1,k})$ conditional on successively $S=0,S=1,\ldots,S=j-1$.  Indeed this is explicitly the distribution of the size biased permutation of $(p_1,\ldots,p_{j-1})$.)
For a summary of the changes in going from $\BX$ and $C(b,n)$ to $\BX'$ and $C'(b,n)$, ball $j$ might move to box $k$, causing $C$ to change by $-1$, $0$, or 1 (i.e., maybe lose a collision in the box $X_j$ where ball $j$ used to land, maybe gain a collision in box $k$) and there might also be one ball, with random label $I$ in the range $1$ to $j-1$, which moves to box $k$, causing an additional change to $C$ by 0 or 1. (Minus 1 is not a possibility, since ball $I$, upon moving to box $k$, causes at least one additional collision.)  The net result is
that our coupling has $C'(b,n)-C(b,n) \in \{-1,0,1,2\}$;  we have a $2$-bounded size bias coupling.  A general principle relating bounded couplings, monotone couplings, and bounded monotone couplings, \cite[Prop.~7.1]{AB}, now implies that there exists a coupling of $C(b,n)$ with its size biased version $C'(b,n)$, for which $0 \le C' - C \le 2$.

Now consider the sum in  \eqref{C.as.sum}, \emph{and} assume that we are in the classical occupancy problem, i.e., that each $X_i$ is \emph{uniformly} distributed on the boxes $1,\ldots,n$.  The event indicated by the summand $W_j$ may be expressed as
\[  
W_j = 1(S>0) \text{ where } S := \sum_{i=1}^{j-1} 1(X_i = X_j).
\]
Thanks to the uniform distributions of the $X_1,\ldots,X_j$, 
the distribution of
$S$ is Binomial$(j-1,\frac{1}{n})$.   As in the previous paragraph, we couple $S$ to $S'$, distributed as $S$ conditional on being strictly positive,  by adding either 0 or 1,  and this lifts to a coupling of $\BX$ with $\BX'$ in which either no ball moves, or else exactly one ball, with random index $I$, moves from a box other than $X_j$, to box $X_j$, where it causes one additional collision.  We have $C'-C \in \{0,1\}$ for all outcomes, i.e., we have a 1-bounded monotone coupling.
\end{proof}

In the setting considered in this paper, the $X_i$'s are uniformly distributed on $1, \ldots, n$, so the proposition provides a 1-bounded monotone coupling of $C'$ with $C$.  Combining this with the main result of \cite{GhoshGold} immediately gives, with 
 $\mu := \e C(b,n)$,
\begin{eqnarray}\label{ghosh}
 \p ( C(b,n) - \mu \le -t  )  &\le&  \exp( -t^2 / ( 2 \mu ) ),\\ 
  \p ( C(b,n) - \mu \ge t  )  &\le&   \exp( -t^2 / ( 2 \mu +  t) ). \nonumber
\end{eqnarray}
for all $t > 0$ and all $b, n$.  This strengthens the Azuma bounds from \eqref{azuma}.

\subsection{Uniform integrability  for $(B(c,n)-\beta(c,n))/\sqrt{n}$.}

\begin{lem}\label{ui2a}
Assume, as in Theorem \ref{MT}, that we are given positive integers $c_1,c_2,\ldots$ with  $\lim_{n \to \infty} c_n = \infty$ and $\lim_{n \to \infty} c_n/n = \alpha_0 \in [ 0, \infty)$.  With  $\beta(c_n,n)$ given by \eqref{def beta}, there exists $n_0 < \infty$ and $\epsilon > 0$ such that for all $n \ge n_0$ and for all $y$,
$$
    \p( | B(c,n) - \beta(c,n) | \ge y\sqrt{n} ) \le \exp( - \min(y, y^2)/104 \, ) .
$$
\end{lem}

\begin{proof}
We check the  bound for $\p(B(c,n) - \beta(c,n) \ge y\sqrt{n})$ with $y \ge 0$; the case of the other sign is comparatively easy and we omit the details. Start from 
\eqref{beta b} and \eqref{middle big display}, and write out explicitly $\e N_0 \equiv \e N_0(b,n) = n \,(1-\frac{1}n)^b$.  This yields
$$ \p(B(c,n)> b)  = \p( N_0(b,n) - \e N_0 < -t)   = \p( C(b,n) - \e C(b,n) < - t)
$$ 
where 
\begin{equation}\label{def t}
  t= n (1-\frac{1}n)^b    -    n e^{-\lambda} + y \sqrt{n}.
\end{equation}
The precise goal 
is to show that there exist $n_0,y_0$ such that for all $n>n_0,y>y_0$,  $t^2/\e C(b,n) \ge \ell(y)$ with $\ell(y)/\log y \to \infty$ as $y \to \infty$.  That is, we want a lower bound on $t^2/\e C(b,n)$ that grows with $y$, and is uniform in $c,n$. 
 The analysis is similar to
that in the proof of Theorem \ref{MT}, 
but we want an inequality, carefully processed to show uniformity.

Using $-\log(1-\frac{1}n) = \frac{1}n + \frac{1}{2n^2} + \frac{1}{3n^3} + \cdots \le \frac{1}n + \frac{1}{n^2}$ for $n \ge 2$, and $e^{z}-1\ge z$ for all real $z$, we have, for $n \ge 2$,
\begin{eqnarray*}
n(1-\frac{1}n)^{b} - ne^{-b/n}  & \ge & n \left( \exp\left(-b \left(1/n + 1/n^2\right)\right) -e^{-b/n}\right)  \\
 & = &n \ e^{-b/n} (\exp(-b/n^2) -1 )  \\
 & \ge & -n \ e^{-b/n} \ \frac{b}{n^2} \ge - \frac{b}n 
\end{eqnarray*}
Using $e^{z}-1\ge z$ again, we have
\[ 
n(e^{-b/n} - e^{-\lambda} ) = n e^{-\lambda} \left( \exp(-y/\sqrt{n})-1 \right) \ge - e^{-\lambda} y \sqrt{n}.
\] 
Adding these two bounds, together with the final term of $t$ from \eqref{def t}, we have
$$
   t \ge - \frac{b}n  + (1 - e^{-\lambda}) y \sqrt{n} = -\lambda - y/\sqrt{n} + (1 - e^{-\lambda}) y \sqrt{n}.
$$
For the delicate case, which is $c\to \infty, c/n \to \alpha_0=0$, we have $\beta^2(c,n) \sim 2cn$ hence $\lambda \to 0$. (The case $\alpha_0>0$ so that $1-e^{-\lambda} \to 1-e^{-\lambda_0}>0$, hence $t \asymp y \sqrt{n}$ is \emph{very easy} in comparison, and can even be handled via Azuma-Hoeffding; we omit further details.) 
With some choice $n_0 \ge 16$, for all $n > n_0$, $1-e^{-\lambda} > \lambda/2$, hence $t \ge  -\lambda - y/\sqrt{n} + \frac{1}2 \lambda y \sqrt{n}$,
and hence for all $y \ge 1$ we have $t \ge   - y/\sqrt{n} + \frac{1}4 \lambda y \sqrt{n}$.  Finally, since $n \lambda = \beta(c,n) \to \infty$, increasing $n_0$ if needed, for all $y \ge 1, n \ge n_0$, we have $t \ge    \frac{1}5 \lambda y \sqrt{n}$.  Squared, $t^2 \ge    \frac{1}{25} \lambda^2 n y^2$, and since 
$\lambda^2 n=\beta^2(c,n)/n \sim 2cn/n = 2c$, increasing $n_0$ if needed, for all $y \ge 1, n \ge n_0$, we have
$t^2 \ge \frac{c}{13} y^2$.

The upper bound \eqref{ghosh}  
has the form:  for $t \ge 0$, $\p(C(b,n) - \e C(b,n) \le -t) \le \exp(-r)$
where
$
  r := t^2/(2 \e C(b,n))    $.  
We are in good shape when 
$\e C(b,n)  \le 4c$, which yields $r \ge y^2/104$.  
(Essentially, this is the main range, with $y ^2 = O(c)$.)

For the remaining cases, where
$y$ is so large that $\e C(b,n) > 4c$, we bypass \eqref{middle big display} and work directly with \eqref{beta b} and
the duality.  Write $\mu := \e C(b,n)$.  With $y \ge 0$ and $b=\beta(c,n)+y \sqrt{n}$ such that $\mu >4c$, hence $t := \mu -c >\frac{1}2 \mu$,
\begin{multline*}
\p(B(c,n) >b) = \p(C(b,n)<c) = \p(C(b,n)-\mu < c - \mu)  \\
\le  \p(C(b,n)-\mu < - \frac{1}2 \mu) \le
\exp(-r) \text{ where } r=\frac{(\mu/2)^2}{2 \mu} = \mu/8.
\end{multline*}
Now in case $y \le n^{2/5}$, using $\beta \sim \sqrt{2cn} = o(n)$, so uniformly in $y \le n^{2/5}$,  $b=o(n)$ and $4c < \mu \sim b^2/(2n) \sim (\sqrt{2c}+y)^2/2$, hence $\mu > 4c$ and
$c \to \infty$ implies $\inf_y y^2/\mu \ge 1/2$ so for sufficiently large $n, \p(B(c,n)>b) \le \exp(- y^2/17)$.  Also, in case $y \ge n^{3/5}$, we have $b > n^{6/5}$ and $B(c,n) \le c+n$, hence, for sufficiently large $n$,  $\p ( B(c,n) > b) = 0$.
 To cover the missing range, if $n^{2/5} \le y \le n^{3/5}$ we simply use $y' = \sqrt{y}$ and $\p(B(c,n) \ge \beta+y \sqrt{n}) \le \p(B(c,n) \ge \beta+y' \sqrt{n})$.
 \end{proof}

\begin{lem}\label{ui2}
Assume, as in Theorem \ref{MT}, that we are given positive integers $c_1,c_2,\ldots$ with  $\lim_{n \to \infty} c_n = \infty$ and $\lim_{n \to \infty} c_n/n = \alpha_0 \in [ 0, \infty)$.  With  $\beta(c_n,n)$ given by \eqref{def beta}, there exists $n_0 < \infty$, such that for every $k=1,2,\ldots$, the family
$$
 \left\{  \left( \frac{B(c_n, n) - \beta(c_n, n)}{\sqrt{n}} \right)^k:  n \ge n_0 \right\}
$$
is uniformly integrable.
\end{lem}
\begin{proof}
As in the proof of Lemma \ref{UI}, the uniform and super-polynomial decaying upper bound from Lemma \ref{ui2a} implies uniform boundedness of the $(k+1)$-st moments for  the $(B(c_n,n)-\beta(c_n,n))/\sqrt{n})$, $n \ge n_0$, which in turn implies uniform integrability of the family of $k$-th powers.
\end{proof}

\section{Moments and variance in the centered case
} \label{moments.duality}

\begin{cor}\label{cor var}
As in Theorem \ref{MT}, suppose $\lim_{n \to \infty} c_n = \infty$ and $\lim_{n \to \infty} c_n/n = \alpha_0 \in [ 0, \infty)$.  Then, with $w$ and $g$ as defined by \eqref{def w} and \eqref{def g}, and with  $\la_0 = w^{-1}(\alpha_0)$,
$$   
     \var  B(c_n,n) \sim n\, g(\lambda_0)^2.
$$
In particular, if $c_n \to \infty$ with $c_n=o(n)$, then $\var B(c_n,n) \sim n/2$.
Furthermore, 
\[
\begin{array}{rclr}
\e [(B(c_n, n) - \beta(c_n, n))^k] &=& o(n^{k/2})&\text{for $k = 1, 3, 5, \ldots$} \\
\e [(B(c_n, n) - \beta(c_n, n))^k] &\sim &(k-1)!! \,g(\lambda_0)^{k} \, n^{k/2}&\text{for $k=2,4,6,\ldots$}
\end{array}
\]
\end{cor}
In the display, $(k-1)!!=(k-1)(k-3)\cdots (5)(3)(1)$.

\begin{proof}
These claims follow from the distributional limit in Theorem \ref{MT}, together with the uniform integrability from Lemma \ref{ui2}, and the moments of the standard normal.
\end{proof}

For the reader's convenience, we re-formulate some of our results for $c = o(n)$:
\begin{cor} \label{on}
If $c = o(n)$, then
\[
\E B(c_n,n) \sim \gamma(c) \sqrt{2cn} \eand \Var B(c_n, n) \sim 2c(1-\gamma(c)^2)\,n
\]
for $\gamma$ as in \eqref{gamma formula 0}.
\end{cor}

\begin{proof}
Combine Corollary \ref{cor var} with equations \eqref{first asymptotic moment} and \eqref{var.limit} and Corollary \ref{var}.
\end{proof}

The claim for the expectation in Corollaries \ref{cor var} and \ref{on} extends \cite[Th.~1]{KuhnStruik} from the regime $c = o(n^{1/4})$ to $c = O(n)$.  Furthermore, when $c = o(n)$, we have:
\[
\lim_{n \to \infty} \frac{\sqrt{\Var B(c, n)}}{\E B(c,n)} = \frac{\sqrt{1 - \gamma(c)^2}}{\gamma(c)} = \frac{1}{2\sqrt{c}} + \frac{1}{32\sqrt{c}^3} - \frac{9}{1024 \sqrt{c}^5} + \cdots
\]
This justifies the following claim made in \cite[p.~221]{KuhnStruik}: ``It turns out that the variance, when compared to the expected [value], is relatively low, especially if the number [$c$] ... is not too small.''

{\small {\subsubsection*{Acknowledgements} SG's research was partially supported by NSF grant DMS-1201542, Emory University, and the Charles T.~Winship Fund.}}

\providecommand{\bysame}{\leavevmode\hbox to3em{\hrulefill}\thinspace}
\providecommand{\MR}{\relax\ifhmode\unskip\space\fi MR }
\providecommand{\MRhref}[2]{%
  \href{http://www.ams.org/mathscinet-getitem?mr=#1}{#2}
}
\providecommand{\href}[2]{#2}

\end{document}